\documentclass[11pt]{amsart}

\usepackage{amssymb,latexsym}
\usepackage{amsthm}
\usepackage{amsmath}
\usepackage{amsfonts}
\usepackage{graphicx}
\usepackage[all]{xy}
\usepackage{fancyhdr}
\usepackage[top=1in,bottom=1in,left=1.25in,right=1.25in]{geometry}

\newtheorem{theorem}{Theorem}[section]
\newtheorem{definition}[theorem]{Definition}
\newtheorem{remark}[theorem]{Remark}
\newtheorem{proposition}[theorem]{Proposition}
\newtheorem{corollary}[theorem]{Corollary}

\def\Q{\mathbb{Q}}
\def\F{\mathbb{F}}
\def\R{\mathbb{R}}
\def\Z{\mathbb{Z}}
\def\A{\mathbb{A}}

\def\C{\mathbb{C}}
\def\G{\mathbb{G}}

\def\I{\mathcal{I}}

\def\M{\mathcal{M}}

\def\Zm{\mathcal{Z}}
\def\V{\mathcal{V}}
\def\vep{\varphi}
\def\veps{\varepsilon}
\def\vp{\varpi}

\def\lra{\longrightarrow}
\def\ra{\rightarrow}
\def\ov{\overline}
\def\ul{\underline}
\def\wh{\widehat}
\def\wt{\widetilde}
\def\st{\stackrel}
\def\tr{\textrm}

\usepackage[pdftex]{hyperref}

\begin{document}
\title{On the cohomology of some simple Shimura varieties with bad reduction}
\author{Xu Shen}
\date{}
\address{Fakult\"at f\"{u}r Mathematik\\
Universit\"at Regensburg\\
Universitaetsstr. 31\\
93040, Regensburg, Germany} \email{xu.shen@mathematik.uni-regensburg.de}

\address{Current address: Morningside Center of Mathematics\\
	No. 55, Zhongguancun East Road\\
	Beijing 100190, China}\email{shen@math.ac.cn}

\renewcommand\thefootnote{}

\begin{abstract}
We determine the Galois representations inside the $\ell$-adic cohomology of some quaternionic and related unitary Shimura varieties at ramified places. The main results generalize the previous works of Reimann and Kottwitz in this setting to arbitrary levels at $p$, and confirm the expected description of the cohomology due to Langlands and Kottwitz.
\end{abstract}

\maketitle
\tableofcontents

\section{Introduction}
This paper is a continuation of our work \cite{Sh}. Based on the methods and results established in \cite{SS} and \cite{Sh}, we determine the Galois representations inside the $\ell$-adic cohomology of some quaternionic and related unitary Shimura varieties at ramified places. The main results generalize the work of Reimann \cite{Re1,Re4} and Kottwitz \cite{Ko2} in this setting to arbitrary levels at $p$ ($\neq \ell$), and confirm the expected description of the cohomology due to Langlands and Kottwitz.

The problem of determining the Galois representations in the cohomology of Shimura varieties, or somehow equivalently, of computing the Hasse-Weil zeta functions of these varieties, has been playing a central role in the Langlands program. Langlands and Kottwitz had given a conjectural description of the Galois representations inside the cohomology, cf. \cite{Ko3}. Roughly it says that, the Galois representation associated to an automorphic representation when restricting to a place above $p$ is given by the local Langlands correspondence for the local reductive group. Many authors have made great contribution to this field. For the related history we refer to the introductions of \cite{Ko3, Ko2}. In \cite{SS} Scholze and Shin have solved this problem for some compact unitary Shimura varieties at (ramified) split places such that the related local reductive groups are products of Weil restriction of $GL_n$. In \cite{Sh} we have also solved the case of Shimura varieties which admit $p$-adic uniformization by finite products of Drinfeld upper half spaces. There the local reductive groups have some factors as the multiplicative group of a central division algebra of invariant $\frac{1}{n}$. Both \cite{SS} and \cite{Sh} allow the level structures at $p$ to be arbitrary.

In this paper, we will concentrate on another special class of Shimura varieties which are related to $GL_2$ (more precisely, the inner and outer forms of $GL_2$). These varieties were already studied previously by Reimann in \cite{Re1, Re2, Re3, Re4} and by Kottwitz \cite{Ko2} (the case $n=2$ there). Let $D$ be a quaternion division algebra over a totally real number field $F$, and $D^\times$ be the associated multiplicative group over $\Q$. After fixing some suitable CM extension $K|F$, one can associate a unitary group $G$ over $\Q$. For any open compact subgroup $C\subset D^\times(\A_f)$, we have the quaternionic Shimura variety $Sh_{D,C}$, which is projective. Unless $D$ is totally indefinite, these varieties are not of PEL type. To study their reduction modulo $p$ and cohomology, in \cite{Re1, Re3} Reimann introduced some PEL Shimura varieties $Sh_{G,\ov{C}}$ for the unitary group $G$ and some related open compact subgroup $\ov{C}\subset G(\A_f)$. By the theory of connected components of Shimura varieties, the varieties $Sh_{D,C}$ can be considered as open and closed subvarieties of some suitable Galois twists of $Sh_{G,\ov{C}}$. Let $p$ be a prime which is unramified in $F$. When $C$ has the form \[C^pC_p\subset D^\times(\A_f^p)\times D^\times(\Q_p)\] with $C_p$ maximal, in loc. cit. Reimann defined some integral models of $Sh_{G,\ov{C}}$ by proposing suitable moduli problems. Then he defined the integral models of $Sh_{D,C}$ by using Galois twists and taking suitable connected components as on generic fibers. By studying the local structures and reduction modulo $p$ of these models, Reimann can determine the local semisimple Hasse-Weil zeta function of $Sh_{D,C}$ as a local semisimple automorphic $L$-function for the above form of $C$ (and also for the case that $C_p\subset D^\times(\Q_p)$ is the Iwahori subgroup, cf. \cite{Re4}). His results generalized the previous works of Langlands \cite{Lan} (in the case $D_{\Q_p}^\times$ unramified) and Rapoport \cite{Ra1} (in the case $D$ is totally indefinite, $p$ is inert in $F$ and $D_{\Q_p}$ is a local quaternion algebra). In this paper we would like to generalize Reimann's results to arbitrary levels at $p$ as in the works of \cite{SS} and \cite{Sh}.

In fact, we can modify the group $G$ a little to get another unitary group $G'$ associated to the CM extension $K|F$ and $D$. For the related levels $\ov{C}'$ to $C$ (see section 2 for precise meaning), we have the Shimura varieties $Sh_{G',\ov{C}'}$, which are the simple Shimura varieties studied by Kottwitz in \cite{Ko2} for $n=2$. It turns out we can treat this case all together with the quaternionic case. The key new point is that here we can take arbitrary $p$, e.g. $D$ can be ramified at the primes of $F$ above $p$ so that the local reductive group $G'_{\Q_p}$ is not quasi-split, therefore not included in \cite{Ko2} or \cite{SS}. Due to the non quasi-split assumption for the reductive groups at $p$, the set of Kottwitz triples is not enough for parameterizing the points on these varieties over finite fields, contrary to the unramified case of \cite{Ko1} and quasi-split case \cite{Sch3}. Here as \cite{Ra2}, \cite{Re1} and \cite{Sh}, we use the original approach of Kottwitz as in \cite{Ko0} to get suitable combinatorial description of the points over finite fields. As in \cite{SS} and \cite{Sh}, having the description of the set of points of reduction modulo $p$ on these varieties, the crucial ingredient is to define some suitable test functions at $p$ which will appear in the trace formula when analyzing the cohomology. There are two cases.
\begin{itemize}
\item For the unitary case, these functions are already available: if $D$ is split at a place $\nu$ above $p$, the factor at $\nu$ was defined in \cite{Sch3} (the EL case $n=2$); if $D$ is ramified at a place $\nu$ above $p$, the factor at $\nu$ was defined in \cite{Sh} (for $n=2$) which in turn was based the approach of Scholze \cite{Sch2, Sch3} by deformation spaces of $p$-divisible groups.
\item For the quaternionic case, one needs just to take the Galois twist into consideration and the method of Scholze applies.\end{itemize}
Then as in \cite{Sh}, the next tasks are to prove
\begin{itemize}
\item some vanishing results for these test functions so that they admit transfers as functions on $D^\times(\Q_p)$ or $G'(\Q_p)$ (cf. \cite{Ra2} conjecture 5.7 and \cite{Ra3} conjecture 10.2),
\item and these transfer functions satisfy some suitable character identities as \cite{SS} conjecture 7.1.
\end{itemize}
For the proofs, we decompose the test functions as products of factors over the places $\nu$ of $F$ above $p$ and consider each factor. Then the split case follows from the results of \cite{SS} and the ramified case follows from the results of \cite{Sh}. At this point, we have to make some assumptions for the cocharacters at these ramified places to put ourselves into the local situation of \cite{Sh}, cf. section 4 for details. Here, in fact, the first point only concerns ramified places.

With these results at hand, one can deduce the desired description of the cohomology. Let $G=D^\times$ or $G'$. Let $l\neq p$ be prime, and $\xi$ be an algebraic $\ov{\Q}_l$-representation of $G$. Then by standard method we can associate $\ov{\Q}_l$-local systems $\mathcal{L}_\xi$ on the Shimura varieties $Sh_{G,C}$ for any open compact subgroup $C\subset G(\A_f)$. Let $E$ be the local reflex field. We are interested in the alternating sum of cohomology
groups
\[H_\xi=\sum_{i}(-1)^i\varinjlim_{C}H^i(Sh_{G,C}\times \ov{\Q}_p,\mathcal{L}_\xi)\]
as a virtual representation of $G(\A_f)\times W_E$. The main theorem is as follows.
Recall the cocharacter $\mu$ associated to the Shimura data gives rise to a representation $r_{\mu}:\, ^L(G_E)\lra GL(V)$ for a finite dimensional $\ov{\Q}_l$-vector space $V$ (cf. \cite{Ko0} 2.1.2).
\begin{theorem}
We have an identity
\[H_\xi=\sum_{\pi_f}a(\pi_f)\pi_f\otimes(r_{\mu}\circ\varphi_{\pi_p}|_{W_E})|-|^{-d/2}\]
as virtual $G(\Z_p)\times G(\A_f^p)\times W_E$-representations. Here $\pi_f$ runs through irreducible admissible representations of $G(\A_f)$, the integer $a(\pi_f)$ is as in \cite{Ko2} p. 657, $\varphi_{\pi_p}$ is the local Langlands parameter associated to $\pi_p$, $d=dimSh_{G,C}$ is the dimension of the Shimura varieties $Sh_{G,C}$ for any open compact subgroup $C\subset G(\A_f)$.
\end{theorem}
In this paper we assume $p$ is unramified in the totally real field $F$ as in the works \cite{Re1,Re2, Re3, Re4}. This assumption will imply the integral models of our Shimura varieties are flat. We remark that the general case which allows the ramification of $p$ should be workable as in \cite{Sch3} by taking flat closure of the generic fibers in the naive integral models.

As a corollary we get the semisimple zeta functions of these Shimura varieties. Our result generalizes the previous works of Reimann \cite{Re1, Re4} to arbitrary levels at $p$. Let $\wt{E}$ be the global reflex field and $\nu$ be a place of $\wt{E}$ above $p$ such that $E=\wt{E}_\nu$.
\begin{corollary}
Let the situation be as in the theorem. Let $C\subset G(\A_f)$ be any sufficiently small compact open subgroup. Then the semisimple local Hasse-Weil zeta function of $Sh_{G,C}$ at the place $\nu$ of $\wt{E}$ is given by
\[\zeta_\nu^{ss}(Sh_{G,C},s)=\prod_{\pi_f}L^{ss}(s-d/2, \pi_p, r_{\mu})^{a(\pi_f)dim\pi_f^C}.\]
\end{corollary}

We briefly describe the content of this article. In section 2, we introduce the quaternionic and related unitary Shimura varieties to be studied. In section 3, we summarize the results of Reimann on integral models of quaternionic Shimura varieties $Sh_{D,C}$ and the unitary Shimura varieties $Sh_{G,\ov{C}}$ and their reduction modulo $p$. We will also treat the simple unitary case $Sh_{G',\ov{C}'}$. In section 4 we define some test functions at $p$ following the method of Scholze for the quaternionic and unitary case and study their key properties. Finally, in section 5 we deduce the cohomology of these Shimura varieties as expected by Langlands and Kottwtiz, and determine their local semisimple zeta functions as products of local semisimple automorphic $L$-functions.\\
 \\
\textbf{Acknowledgments.} I would like to thank Peter Scholze sincerely for his encouragement and suggestions. I want to thank Michael Rapoport for a useful remark. I should also thank the referee for careful reading and valuable suggestions. This work started while the author was a postdoc at the Mathematical Institute of the University of Bonn, and was finished after the author moved to the University of Regensburg. This work was supported by the SFB/TR 45 ``Periods, Moduli Spaces and Arithmetic of Algebraic Varieties'' and the SFB 1085 ``Higher Invariants'' of the DFG.

\section{Some simple Shimura varieties}
In this section we will mainly (but not all) follow the notations of \cite{Re1} section 1. We will also introduce some related unitary Shimura varieties which are special cases of those studied in \cite{Ko2} but not included in the works \cite{Re1, Re2, Re3, Re4}.

Let $D$ be a quaternion division algebra over a totally real number field $F$ and let $\mathcal{Z}$ be the set of infinite places of $F$ at which $D$ is split. We assume $\mathcal{Z}$ is not empty and for each $v\in \mathcal{Z}$, we fix an embedding of $\R$-algebras $\C\subset D_v=D\otimes_{F,v}\R$. Let $D^\times$ denote the reductive group over $\Q$ associated to $D$. We get a homomorphism of algebraic groups over $\R$
\[\begin{split}h_D: \mathbb{S}=\C^\times&\longrightarrow \prod_{v\in\Zm}D_v^\times\subset (D_\R)^\times\\
z&\longmapsto (z|v\in\Zm).\end{split}\]
Let $X$ be the conjugacy class of $h_D$ which does not depend on the choice of embeddings $\C\subset D_v$. $(D^\times, X)$ forms a Shimura datum with reflex field
\[E(D)=\Q(\sum_{v\in\Zm}v(x)|\,x\in\,F)\subset \C.\]
For every open compact subgroup $C\subset D^\times(\A_f)$, let $Sh_{D,C}$ be the projective Shimura variety with level $C$ over $E(D)$. Except for the special case that $\Zm=Hom(F,\R)$, i.e. $D$ is totally indefinite, these Shimura varieties are not of PEL type. To construct their integral models and to study the reduction modulo $p$, we introduce some related Shimura varieties as follows.

Fix a totally imaginary quadratic extension $K|F$ with complex conjugation $c$ and a subset $\V_K$ of $Hom(K,\C)$ such that its restriction to $F$ induces a bijection $\V_K\ra\V:=Hom(F,\R)-\Zm$. For each $v\in\V_K$, fix an isomorphism $K_v\simeq \C$. Consider $K^\times$ as an algebraic group over $\Q$ together with the homomorphism of real algebraic groups
\[\begin{split}h_K: \mathbb{S}&\longrightarrow \prod_{v\in\V_K}K_v^\times\subset (K\otimes\R)^\times\\
z&\longmapsto (z|v\in\V_K).\end{split}\]For every open compact subgroup $C\subset K^\times(\A_f)$ we get a Shimura variety $Sh_{K,C}$ over the reflex field $E(K)$ which is a finite extension of $E(D)$. This variety is a finite reduced scheme with
\[Sh_{K,C}(\ov{\Q})=K^\times(\A_f)/CK^\times.\]
The Galois action of $Gal(\ov{\Q}|E(K))$ on these varieties is understood by class field theory and the reciprocity law map
\[r_{K}: E(K)^\times\lra K^\times.\]

Let $B=D\otimes_FK$ be the semisimple algebra over $K$. We choose a positive involution $\ast$ on $B$ and fix a non-degenerate alternating $\Q$-bilinear form $\psi: B\times B\ra \Q$. Consider the associated unitary similitude group $G$ over $\Q$ consisting of automorphisms of $B$ which preserve $\psi$ up to an $F^\times$-scalar. Then there is an exact sequence for algebraic groups over $\Q$ (cf. \cite{Re1} section 1)
\[1\lra F^\times \lra D^\times\times K^\times \st{\tau}{\lra} G\lra 1.\]Here $F^\times\lra D^\times\times K^\times$ is the embedding $f\mapsto (f,f^{-1})$. In particular, the group $G$ is the Weil restriction of a reductive group defined over $F$. Under the above exact sequence the center of $G$ is isomorphic to $K^\times$. Let $Y$ be the conjugacy class of
\[h_G=\tau_{\R}\circ (h_D\times h_K): \mathbb{S}\longrightarrow G_\R.\] Then $(G,Y)$ defines a Shimura datum with reflex field $E(K)$. For each open compact subgroup $C\subset G(\A_f)$, we have a projective variety $Sh_{G,C}$ over $E(K)$, which is a coarse moduli space of abelian varieties with some additional structures. For the reader's convenience, we review the related moduli problem as follows. To do this, let us first review a definition of \cite{Re1}. In the following, for a number filed $F'$ we will denote by $O_{F'}$ its ring of integers. Since the Galois group $Gal(\ov{\Q}|E(K))$ stabilizes $\V_K$ and $\Zm$, there is a natural decomposition
\[O_K\otimes_\Z O_{E(K)}=O(\Zm)\times O(\V_K)\times \ov{O}(\V_K),\]where
\[\begin{split}
O(\Zm)&=\{ x\in O_K\otimes_\Z O_{E(K)}|\,(\ov{\varphi}\otimes id)(x)=(\varphi\otimes id)(x)=0, \,\tr{if} \,\varphi\in \V_K \},\\
O(\V_K)&=\{x\in O_K\otimes_\Z O_{E(K)}| \,(\varphi\otimes id)(x)=0, \,\tr{if} \,\varphi\in Hom(K,\C)-\V_K\},\\
\ov{O}(\V_K)&=\{ x\in O_K\otimes_\Z O_{E(K)}|\, (\ov{\varphi}\otimes id)(x)=0,  \,\tr{if} \,\varphi\in Hom(K,\C)-\V_K \}.
\end{split}\]
Then if $\mathfrak{F}$ is a sheaf of $O_K\otimes O_{E(K)} $-modules on an $O_{E(K)}$-scheme $S$, the above decomposition induces a corresponding decomposition \[\mathfrak{F}=\mathfrak{F}(\Zm)\times \mathfrak{F}(\V_K)\times \ov{\mathfrak{F}}(\V_K).\]Let $L$ be a CM extension of $F$ which is contained in $D$. By \cite{Re1} definition 2.1, a sheaf of type $(L,\V_K)$ on an $O_{E(K)}$-scheme $S$ is a coherent sheaf of $\mathcal{O}_S$-modules $\mathfrak{F}$ together with a homomorphism \[O_L\otimes_{O_F}O_K\ra End_{\mathcal{O}_S}\mathfrak{F}\] such that locally for the flat topology on $S$ we have
\begin{itemize}
	\item $\mathfrak{F}$ is a free $O_L\otimes_\Z \mathcal{O}_S$-module,
	\item $\mathfrak{F}(\Zm)$ is a free $O_L\otimes_{O_F}O(\Zm)\otimes_{O_{E(K)}} \mathcal{O}_S$-module,
	\item $\ov{\mathfrak{F}}(\V_K)=0$.
\end{itemize}
 Now consider the moduli problem $M_{G,C}$ such that for any locally noetherian $E(K)$-scheme $S$, $M_{G,C}(S)$ is the set of isomorphism classes of $(A,\iota, \Lambda, \ov{\eta})$ where
\begin{itemize}
\item $A$ is a projective abelian scheme over $S$,
\item $\iota: B\ra End(A)\otimes\Q$ is an embedding such that $LieA$ is a sheaf of type $(L,\V_K)$, where $L\subset D$ is a fixed CM extension of $F$ contained in $D$,
\item $\Lambda=\lambda\circ\iota(F^\times)\subset Hom(A, A^\vee)\otimes\Q$ for a polarization $\lambda$ of $A$,
\item $\ov{\eta}$ is a $C$-level structure on $A$.
\end{itemize}
Then $Sh_{G,C}$ is the coarse moduli space for $M_{G,C}$. In fact, $Sh_{G,C}$ represents the \'etale sheafification of the functor $M_{G,C}$ (cf. \cite{Re3} definition 2.2 and proposition 2.14).
The above exact sequence defines a morphism over $E(K)$ for $C_1\times C_2\subset D^\times(\A_f)\times K^\times(\A_f)$ mapping to $C\subset G(\A_f)$
\[(Sh_{D,C_1}\times E(K))\times Sh_{K,C_2}\lra Sh_{G,C},\]
when $C_1, C_2$ vary these morphisms are $D^\times(\A_f)\times K^\times(\A_f)$-equivariant.

For every open compact subgroup $C\subset D^\times(\A_f)$, we define the following notations:
\[\begin{split}&C_F=C\cap F^\times(\A_f),\\
&N_C=min\{N\geq 0|(1+NO_F\otimes\wh{\Z})^\times\subset C_F\},\\
&C_K=C_F(1+N_CO_K\otimes\wh{\Z})^\times,\\
&\ov{C}=\tau_{\A_f}(C\times C_K).\end{split}\]
Then as in \cite{Re1} section 1, for any $C\subset D^\times(\A_f)$ with $N_C\geq 3$, we get a Galois covering
\[(Sh_{D,C}\times E(K))\times Sh_{K,C_K}\lra Sh_{G,\ov{C}}\]with Galois group $F^\times(\A_f)/C_FF^\times$. This yields a constant Galois covering with the same Galois group
\[(Sh_{D,C}\times E(K))\times K^\times(\A_f)/C_KK^\times\lra M_{\ov{C}},\]
where the projective variety $M_{\ov{C}}$ over $E(K)$ is defined as follows. The above morphism $\tau: D^\times\times K^\times \ra G$ identifies $K^\times$ with the center of $G$ and $C_K$ with a subgroup of $\ov{C}$, hence defines an action of $K^\times(\A_f)/C_KK^\times$ on $Sh_{G,\ov{C}}$. Let $L_C$ denote the finite abelian extension of $E(K)$ such that
\[r_K: E(K)^\times(\A_f)\lra K^\times(\A_f)\]induces an injection
\[r_{K,C}: Gal(L_C|E(K))\lra K^\times(\A_f)/C_KK^\times.\]
Then there is an isomorphism $\phi_C$ of $L_C$-schemes such that the following diagram
\[\xymatrix{
M_{\ov{C}}\times L_C\ar[r]^{\phi_C}\ar[d]^{id\times\sigma}& Sh_{G,\ov{C}}\times L_C\ar[d]^{r_{K,C}(\sigma)^{-1}\times\sigma}\\ M_{\ov{C}}\times L_C\ar[r]^{\phi_C}& Sh_{G,\ov{C}}\times L_C
}\]
commutes for every $\sigma\in Gal(L_C|E(K))$. The existence and uniqueness of $M_{\ov{C}}$ follows from descent theory. Let \[\pi_C: M_{\ov{C}}\lra K^\times(\A_f)/F^\times(\A_f)C_KK^\times\]be the projection, then there is a natural isomorphism of $E(K)$-schemes
\[Sh_{D,C}\times E(K)\lra \pi_C^{-1}(1)\subset M_{\ov{C}}.\]We can define an action of $G(\A_f)$ on the varieties $M_{\ov{C}}$ when $C$ varies. Then the above isomorphisms are $D^\times(\A_f)$-equivariant with respect to $\tau$.
In fact one can interpret $M_{\ov{C}}$ as a Shimura variety as follows.
\begin{proposition}
$(G,h_Gh_K^{-1})$ defines a Shimura datum with reflex field $E(D)$, where $h_Gh_K^{-1}: \mathbb{S}\lra G$ is defined by $\tau_\R\circ (h_D\times h_Kh_K^{-1})$. Let $\ov{C}\subset G(\A_f)$ be the open compact subgroup as above, and $Sh^{t}_{G,\ov{C}}$ be the associated Shimura variety with level $\ov{C}$ for the datum $(G,h_Gh_K^{-1})$. Then we have a canonical isomorphism of varieties over $E(K)$
\[M_{\ov{C}}\simeq Sh^{t}_{G,\ov{C}}\times E(K).\]
\end{proposition}
\begin{proof}
This is clear from our construction and the basic theory of canonical models of Shimura varieties.
\end{proof}

We would like to introduce a further class of related Shimura varieties. These are of special cases studied by Kottwitz in \cite{Ko2}. Consider the norm map of algebraic groups over $\Q$
\[\begin{split}G&\lra F^\times\\ (g,z)&\longmapsto Nm(g)zz^c,\end{split}\]
where $Nm: D^\times \ra F^\times$ is the reduced norm. We consider $\Q^\times(=\G_m)\subset F^\times(=Res_{F|\Q}\G_m)$, and let $G'\subset G$ be the inverse image of $\Q^\times$ under the above norm map. It is also the group of $B$-module automorphisms of $B$ which preserve $\psi$ up to a $\Q^\times$-scalar. The center of $G'$ is $K^\times\cap G'$. The morphism $h_G$ factors through $G'$ and defines \[h_{G'}: \mathbb{S}\lra G'_{\R}.\]$G'$ and the conjugacy class of $h_{G'}$ then define a Shimura datum with reflex field $E(K)$. We note that since $G'$ is associated to a quaternion algebra, it satisfies the Hasse principle for $H^1(\Q,G')$, cf. \cite{Ko1} section 7. For any open compact subgroup $C\subset G'(\A_f)$, we have a Shimura variety $Sh_{G',C}$ over $E(K)$, which is a moduli space of abelian varieties with some additional structures when $C$ is sufficiently small. For any locally noetherian $E(K)$-scheme $S$, $Sh_{G',C}(S)=\{(A,\iota,\lambda,\ov{\eta})\}/\simeq$ with $(A,\iota,\ov{\eta})$ the same as in the case of $Sh_{G,C}$ and $\lambda$ is a principal polarization of $A$. For any open compact subgroup $C\subset D^\times(\A_f)$, we have defined the open compact subgroup $\ov{C}\subset G(\A_f)$. Let $\ov{C}'=\ov{C}\cap G'(\A_f)$, which is an open compact subgroup of $G'$. We have a morphism of varieties over $E(K)$
\[Sh_{G',\ov{C}'}\lra Sh_{G, \ov{C}}.\]
In fact, the Shimura varieties $Sh_{D, C},\, Sh_{G, \ov{C}},\, Sh_{G',\ov{C}'}$ have isomorphic geometric irreducible components, and we can also use $Sh_{G',\ov{C}'}$ to study the geometry of $Sh_{D,C}$.

\section{Integral models and points of reduction modulo $p$}
We will define some integral models of the Shimura varieties introduced in last section. Since our ultimate goal is to study their $\ell$-adic cohomology, we will describe the points of reduction modulo $p$ of these varieties. For the Shimura varieties $Sh_{D, C}$ and $Sh_{G, \ov{C}}$, we summarize the results of sections 2-7 of \cite{Re1}. We note that the authors of \cite{TX} have also studied the geometry of these Shimura varieties, but there they made some additional assumptions so that they got smooth integral models. Here we allow arbitrary ramifications, and the integral models obtained are all normal, projective and flat.

Let $p$ be a prime which is unramified in $F$. The primes of $F$ above $p$ are denoted by $\nu_1,\dots,\nu_m$. Assume that each $\nu_i$ splits in $K$ as $\tilde{\nu}_i,\tilde{\nu}_i^c$. Fix a prime $\vp$ of $E(K)$ above $p$, and let $E=E(K)_\vp$. By abuse of notation, the base changes of $Sh_{D, C},\, Sh_{G, \ov{C}},\, Sh_{G',\ov{C}'}$ over $E$ are still denoted by the same notation. Let $O_E$ be the integer ring of $E$. Throughout this set, $C$ is of the form $C=C^pC_p^0\subset D^\times(\A_f^p)\times D^\times(\Q_p)$ with $C_p^0$ the maximal open subgroup of $D^\times(\Q_p)$ associated to a maximal order $O_D\subset D_{\Q_p}$. We are going to define some integral models of these varieties over $O_E$ and describe their points of reduction modulo $p$. We require the conditions (2.6)-(2.9) in section 2 of \cite{Re1} for the choices of $L, K, \V_K,$ and $\psi$ hold true.

Let us begin by the easy case of $Sh_{G',\ov{C}'}$. We assume $C^p$ is sufficiently small so that $C$ and $\ov{C}'$ are sufficiently small (actually the $C^p$ such that $N_C\geq 3$ will be enough). Corresponding to $O_D$ we have a maximal oder $O_B\subset B_{\Q_p}$. Then there is a projective $O_E$-scheme $S_{G',\ov{C}}$ such that for each locally noetherian connected $O_E$-scheme $S$, $S_{G',\ov{C}'}(S)$ is the set of isomorphism classes of objects $(A,\iota,\lambda,\ov{\eta})$ where
\begin{itemize}
\item $A$ is a projective abelian scheme over $S$,
\item $\iota: O_B\ra End(A)$ is an embedding such that $LieA$ is a sheaf of type $(L,\V_K)$,
\item $\lambda$ is a $p$-principal polarization of $A$,
\item $\ov{\eta}$ is a $(\ov{C}')^p$-level structure on $A$.
\end{itemize}
Since we assume $p$ is unramified in $F$, the theory of local models tells us that the scheme $S_{G',\ov{C}}$ is flat over $O_E$ (cf. \cite{G} theorem 4.6.1). Note the local reductive group has the form
\[G'_{\Q_p}\simeq \prod_{i=1}^mB_{\tilde{\nu}_i}^\times\times\G_m\simeq\prod_{i=1}^mD_{\nu_i}^\times\times \G_m.\]

To define the integral model of $Sh_{G,\ov{C}}$, one can consider the integral version $\mathcal{M}_{G,\ov{C}}$ of the moduli problem $M_{G,\ov{C}}$. More precisely, for any locally noetherian connected $O_E$-scheme $S$, $\mathcal{M}_{G,\ov{C}}(S)$ is the set of isomorphism classes of $(A,\iota, \Lambda, \ov{\eta})$ where
\begin{itemize}
\item $A$ is a projective abelian scheme over $S$,
\item $\iota: O_B\ra End(A)\otimes\Z_{(p)}$ is an embedding such that $LieA$ is a sheaf of type $(L,\V_K)$,
\item $\Lambda=\lambda\circ\iota(F^\times)\subset Hom(A, A^\vee)\otimes\Q$ for a $p$-principal polarization $\lambda$ of $A$,
\item $\ov{\eta}$ is a $\ov{C}^p$-level structure on $A$.
\end{itemize} By \cite{Re1} proposition 2.14, there is a coarse moduli scheme $S_{G,\ov{C}}$ over $O_E$ for the moduli problem $\mathcal{M}_{G,\ov{C}}$, which is projective flat.
In fact $S_{G,\ov{C}}$ represents the \'etale sheafification of the moduli problem $\M_{G,C}$ (cf. remark 2.17 of \cite{Re1}). It is important for us to note that there is an action of $K^\times(\Q_p)$ on $S_{G,\ov{C}}$, which extends the action of $K^\times(\Q_p)$ on $Sh_{G,\ov{C}}$. This action can be described as follows. An element $k\in K^\times(\Q_p)$ sends a point $(A,\iota,\Lambda,\ov{\eta})$ to $(A,\iota,\Lambda,\ov{\eta k})$. The local reductive group has the form
\[G_{\Q_p}=\prod_{i=1}^mG_i,\]where for each $1\leq i\leq m$ the group $G_i$ is defined by the following exact sequence of reductive groups over $\Q_p$
\[1\lra F_{\nu_i}^\times\lra D_{\nu_i}^\times\times K_{\tilde{\nu}_i}^\times\times K_{\tilde{\nu}^c_i}^\times\lra G_i\lra 1.\]
In particular, after fixing isomorphisms $K^\times_{\wt{\nu_i}}\simeq K^\times_{\wt{\nu_i}^c}\simeq F^\times_{\nu_i}$, we can identity $G_i$ with the group $D_{\nu_i}^\times\times F_{\nu_i}^\times$. Let $\ov{\F}_p$ be a fixed algebraic closure of the finite field $\F_p$. In \cite{Re1} Reimann computed the strict complete local rings $\wh{O}_x$ for any $\ov{\F}_p$-point $x$ of $S_{G,\ov{C}}$, and proved the scheme $S_{G,\ov{C}}$ is normal.

The reciprocity law $r_K$ of the finite Shimura varieties $Sh_{K,C_K}$ gives rise to a continuous homomorphism
\[r_{K,p}: E^\times \lra K^\times(\Q_p).\]Let $k_D=r_{K,p}(p)$, then it acts on $S_{G,\ov{C}}$. This element can be described explicitly as \[k_D=((p^{e_1},1),\cdots,(p^{e_m},1))\in K^\times(\Q_p)=\prod_{i=1}^m(K_{\wt{\nu_i}}^\times\times K^\times_{\wt{\nu_i}^c})\simeq\prod_{i=1}^m(F_{\nu_i}^\times\times F_{\nu_i}^\times),\]where for $1\leq i\leq m$, the integers $e_i$ are defined as in lemma 5.9 of \cite{Re1}. Let $Fr\in Gal(\Q_p^{nr}|E)$ denote the Frobenius automorphism and $\mathcal{M}_{\ov{C}}$ the normal projective flat $O_E$-scheme which is, up to isomorphism, uniquely determined by the existence of an isomorphism $\phi_C$ of $\Z_p^{nr}$-schemes such that the following diagram commutes:
\[\xymatrix{
\M_{\ov{C}}\times \Z_p^{nr}\ar[r]^{\phi_C}\ar[d]^{id\times Fr}& S_{G,\ov{C}}\times \Z_p^{nr}\ar[d]^{k_D\times Fr}\\ \M_{\ov{C}}\times \Z_p^{nr}\ar[r]^{\phi_C}& S_{G,\ov{C}}\times \Z_p^{nr}
}\]
In fact $\M_{\ov{C}}$ and $S_{G,\ov{C}}$ become isomorphic to each other after a finite unramified extension $\Z_{p^d}$ of $\Z_p$, as some finite power of $k_D$ will act trivially on $Sh_{K,C_K}$ and thus on $S_{G,\ov{C}}$. Here $\Z_p^{nr}$ (resp. $\Z_{p^d}$) is the integer ring of the maximal (resp. degree $d$) unramified extension of $\Q_p$.
When $C^p$ varies, there is a continuous action of $K^\times(\Q_p)\times G(\A_f^p)$ on the tower $\M_{\ov{C}}$. And we have natural $K^\times(\Q_p)\times G(\A_f^p)$-equivariant isomorphisms of varieties over $E$
\[\M_{\ov{C}}\times E\simeq M_{\ov{C}}\times E.\]
The scheme $\M_{\ov{C}}$ is normal. We have a $K^\times(\A_f)$-equivariant morphism of $O_E$-schemes
\[\pi_C: \M_{\ov{C}}\lra K^\times(\A_f)/F^\times(\A_f)C_KK^\times.\]We define an integral model of $Sh_{D, C}$ by
\[S_{D,C}:=\pi_C^{-1}(1)\subset \M_{\ov{C}}.\]
When $C^p$ varies, we get a tower of normal projective flat $O_E$-schemes $S_{D,C}$ with an action of $D^\times(\A_f^p)$. The local reductive group for this case is
\[ D^\times_{\Q_p}=\prod_{i=1}^mD^\times_{\nu_i}.\]

We want a description of the set of $\ov{\F}_p$-points on $S_{G,\ov{C}}$. In \cite{Re1} definition 4.7, there is a set $\I_D$ depending on $D$, which is roughly the set of all isogeny types of the abelian varieties with additional structures. For each $\vep\in \I_D$, there is a reductive group $G_\vep$ over $\Q$ and a set $Y_\vep$ together with an automorphism $Fr_\vep: Y_\vep\ra Y_\vep$ such that the subset $S_{G,\ov{C}}(\vep)\subset S_{G,\ov{C}}(\ov{\F}_p)$ consisting of elements of isogeny type $\vep$ can be written as
\[G_\vep(\Q)\setminus Y_\vep\times G(\A_f^p)/\ov{C}^p,\]with the Frobenius $Fr$ acts via the automorphism $Fr_\vep$ of $Y_\vep$. In fact, let $x_0=(A,\iota,\Lambda,\ov{\eta})\in S_{G,\ov{C}}(\ov{\F}_p)$ be a fixed point of isogeny type $\vep$. One can take $G_\vep=Aut(A,\iota,\Lambda)$. Let $N=D(A[p^\infty])_\Q$ be the rational covariant Dieudonn\'e module of the associated $p$-divisible group, which is a $B$-module equipped with a perfect hermitian form. Let $\mathcal{F}$ denote the Frobenius automorphism on $N$. Then $Y_\vep$ is the set of all $O_B$-invariant lattices $M\subset N$ such that $pM\subset \mathcal{F}M\subset M$, $\mathcal{F}M/pM$ is an $O_L\otimes O_K\otimes\ov{\F}_p$-module of type $(L,\V_K)$, and there exists a perfect paring on $M$. In other words, $Y_\vep$ parameterizes all the Dieudonn\'e modules with additional structures coming from the points in $S_{G,\ov{C}}(\ov{\F}_p)$ which are $p$-isogenous to $x_0$. After fixing a lattice of the form $O_B\otimes W(\ov{\F}_p)$, we can identity $Y_\vep$ with a subset of $G(\mathcal{K})/G(O_{\mathcal{K}})$ where $\mathcal{K}=W(\ov{\F}_p)_\Q$, and $O_{\mathcal{K}}=W(\ov{\F}_p)$. Here $G(O_{\mathcal{K}})$ is defined by the integral model of $G$ associated to $O_B$. The actions of $G_\vep(\Q)$ on $Y_\vep$ and $G(\A_f^p)/\ov{C}^p$ occurring in the above quotient are defined as follows. There is an element $b=b(\vep)\in G(\mathcal{K})$ associated to $\vep$ well defined up to $\sigma$-conjugacy, such that the Frobenius $\mathcal{F}$ acts on $N$ as $b\sigma$. Here, as usual $\sigma$ is the Frobenius morphism on $G(\mathcal{K})$. Then we can define a reductive group $J_b$ over $\Q_p$ such that \[J_b(\Q_p)=\{g\in G(\mathcal{K})|\,(b\sigma)g=g(b\sigma)\}.\]Then $J_b(\Q_p)$ acts on $Y_\vep$. We have an embedding \[G_\vep(\Q)\subset J_b(\Q_p)\] which gives the action of $G_\vep(\Q)$ on $Y_\vep$. On the other hand, we have an embedding \[G_\vep(\A_f^p)\subset G(\A_f^p),\]together with the natural embedding $G_\vep(\Q)\subset G_\vep(\A_f^p)$, giving the action of $G_\vep(\Q)$ on $G(\A_f^p)/\ov{C}^p$. We have
\[S_{G,\ov{C}}(\ov{\F}_p)\simeq\coprod_{\vep\in\I_D}G_\vep(\Q)\setminus Y_\vep\times G(\A_f^p)/\ov{C}^p.\]When the level $C^p$ varies, the above isomorphism is $K^\times(\Q_p)\times G(\A_f^p)$-equivariant.

For each $\vep\in \I_D$, there is a reductive group $H_\vep$ over $\Q$ and a set $X_\vep$ together with an automorphism $Fr_{\vep}: X_\vep\ra X_\vep$ such that the subset $S_{D,C}(\vep)$ of $\vep$-isogeny type points can be written as
\[H_\vep(\Q)\setminus X_\vep\times D^\times(\A_f^p)/C^p,\]with the Frobenius $Fr$ acts via the automorphism $Fr_{\vep}$ of $X_\vep$. For the definition of $H_\vep$, see definition 4.6 of \cite{Re1}. The set $X_\vep$ is defined as \[\{d\in D^\times(\mathcal{K})|d^{-1}d_\vep\sigma(d)\in Y_D\}/(O_D\otimes W(\ov{\F}_p))^\times,\]where 
\begin{itemize}
\item $Y_D$ is as in definition 5.1 of \cite{Re1}: it is set of all $d\in D^\times (\mathcal{K})$ such that $d^{-1}$ is contained in $O_D\otimes W(\ov{\F}_p)$ and for every $v\in Hom(F,\mathcal{K})\simeq Hom(F,\R)$, 
\[ord_p\Big((v\otimes id_{\mathcal{K}})\big(Nm(d)\big)\Big)=\left\{\begin{array}{ll}
-1& \tr{if}\, v\in \Zm,\\
0 & \tr{if} \, v\in \V.
\end{array}  \right.\]
\item $d_\vep$ is a fixed choice of element defined by lemma 5.2 of \cite{Re1}.
\end{itemize}
The set $X_\vep$ can also be defined as a set of some suitable Dieudonn\'e lattices, cf. \cite{Re1} lemma 5.10. Similar to the above, there exists a reductive group $J_b$ over $\Q_p$ with $J_b(\Q_p)$ acts on $X_\vep$. The group actions occurring in the quotient are defined through embeddings $H_\vep(\Q)\subset J_b(\Q_p)$ and $H_\vep(\Q)\subset H_\vep(\A_f^p)\subset D^\times(\A_f^p)$.
We have an $F^\times(\Q_p)\times D^\times(\A_f^p)$-equivariant bijection (cf. \cite{Re1} theorem 6.6 (ii))
\[S_{D,C}(\ov{\F}_p)\simeq \coprod_{\vep\in\I_D}H_\vep(\Q)\setminus X_\vep\times D^\times(\A_f^p)/C^p.\]
In fact, this is deduced from the description of the set $S_{G,\ov{C}}(\ov{\F}_p)$ above in the following way: by lemma 4.10 of loc. cit., there is an exact sequence of reductive groups over $\Q$
 \[1\lra F^\times\lra H_\vep\times K^\times\lra G_\vep\lra 1,\] and by lemma 5.9 of \cite{Re1} there is a $G_\vep(\Q_p)$-equivariant bijection
 \[Y_\vep\simeq F^\times(\Q_p)\setminus\big( X_\vep\times K^\times(\Q_p)\big)/(O_K\otimes\Z_p)^\times,\]
 where an element $f\in F^\times(\Q_p)$ acts on $X_\vep$ via its natural embedding into $H_\vep(\Q_p)$ and on $K^\times(\Q_p)$ as multiplication by $f^{-1}$. The Frobenius $Fr_\vep$ on the left hand side acts as $Fr_\vep\times k_D^{-1}$ on the right hand side. Therefore,
 for each $\vep\in \I_D$ we can rewrite $S_{G,\ov{C}}(\vep)$ as
\[\begin{split}S_{G,\ov{C}}(\vep)&\simeq \big(H_\vep(\Q)\times K^\times\big)\setminus \Big(F^\times(\Q_p)\setminus\big( X_\vep\times K^\times(\Q_p)\big)/(O_K\otimes\Z_p)^\times\Big)\times G(\A_f^p)/\ov{C}^p\\ &\simeq F^\times(\A_f)\setminus\Big(\big(H_\vep(\Q)\setminus X_\vep\times D^\times(\A_f^p)/C^p\big)\times K^\times(\A_f)/C_KK^\times\Big),\end{split}\]with the Frobenius $Fr_\vep$ acts via the endomorphism $Fr_\vep\times k_D^{-1}$ of $X_\vep\times K^\times(\Q_p)$.

Now we describe the set $S_{G',\ov{C}'}(\ov{\F}_p)$. There is a reductive group $I_\vep$ over $\Q$ and a set $Z_\vep$ together with an automorphism $Fr_{\vep}: Z_\vep\ra Z_\vep$ such that the subset $S_{G',\ov{C}'}(\vep)$ of points of isogeny type of $\vep$ can be written as
\[I_\vep(\Q)\setminus Z_\vep\times G'(\A_f^p)/(\ov{C}')^p,\]with the Frobenius acts via the automorphism $Fr_\vep$ of $Z_\vep$. The group $I_\vep$ and the set $Z_\vep$ are defined similarly to the case of $S_{G,\ov{C}}$. As always, after fixing some suitable basis we can identify $Z_\vep$ with a subset of $G'(\mathcal{K})/G'(O_{\mathcal{K}})$. Also, there exists a reductive group $J_b$ over $\Q_p$ with $J_b(\Q_p)$ acts on $Z_\vep$. The group actions occurring in the quotient are defined through embeddings $I_\vep(\Q)\subset J_b(\Q_p)$ and $I_\vep(\Q)\subset I_\vep(\A_f^p)\subset G'(\A_f^p)$. We have the description of $\ov{\F}_p$-points
\[S_{G',\ov{C}'}(\ov{\F}_p)\simeq \coprod_{\vep\in\I_D}I_\vep(\Q)\setminus Z_\vep\times G'(\A_f^p)/(\ov{C}')^p.\]
\begin{proposition}Under the above notations, we have a natural morphism of schemes over $O_E$
\[S_{G',\ov{C}'}\lra S_{G,\ov{C}},\]which induces a map on the sets of $\ov{\F}_p$-points
$f: S_{G',\ov{C}'}(\ov{\F}_p)\lra S_{G,\ov{C}}(\ov{\F}_p).$ Under this map we have equalities for each $\vep\in\I_D$
\[S_{G',\ov{C}'}(\vep)=f^{-1}(S_{G,\ov{C}}(\vep)).\]Moreover, $S_{G',\ov{C}'}$ is normal.
\end{proposition}
\begin{proof}
We have the natural morphism of functors $S_{G',\ov{C}'}\lra \M_{G,\ov{C}}$, which sends $(A,\iota,\lambda,\ov{\eta})$ to $(A,\iota,\Lambda,\ov{\eta})$. Moreover, it induces an embedding of a finite \'etale quotient of $S_{G',\ov{C}'}$ into $\M_{G,\ov{C}}$, cf. the proof of proposition 2.14 in \cite{Re1}. Now the statements in the proposition are clear.
\end{proof}

For any $\vep\in\I_D$, the sets $S_{G',\ov{C}'}(\vep),\,S_{G,\ov{C}}(\vep)$ can be also understood from the point of view of the associated Rapoport-Zink spaces, \cite{RZ}. Namely, the set $Z_\vep$ can be considered as the set of $\ov{\F}_p$-points of the Rapoport-Zink space associated to the local PEL data (cf. \cite{RZ} 3.18) \[(B_{\Q_p},\ast,V=B_{\Q_p},\psi, O_{B_{\Q_p}}, \Lambda=O_{B_{\Q_p}}, b,\{\mu\}),\] where $b\in G'(\mathcal{K})$ is an element up to $\sigma-G'(O_\mathcal{K})$ conjugacy defined by the $p$-divisible group associated to a fixed point $x\in S_{G',\ov{C}'}(\vep)$, $\{\mu\}$ is the conjugacy class of $\ov{\Q}_p$-cocharacters of $G'$ with field of definition $E$, defined from the morphism $h_{G'}$ in the Shimura datum. The case for $Y_\vep$ is similar, cf. \cite{BZ}. In fact, the sets $Z_\vep$ and $Y_\vep$ are examples of affine Deligne-Lusztig varieties in our setting, which generally occur in the description of mod $p$ points of Shimura varieties.

We are more interested in the Shimura varieties $Sh_{D,C},\, Sh_{G',\ov{C}'}$. To compute their $\ell$-adic cohomology, we need also to understand the points on them over finite fields. Fix an integer $j\geq 1$. Let $\kappa_E$ be the residue field of $E$, and set $r=j[\kappa_E:\F_p]$. Since we assume that $p$ is unramified in $F$, the extension $E|\Q_p$ is also unramified.  We begin by recall the result of \cite{Re1} for the case of $Sh_{D,C}$.
\begin{proposition}
There is a bijection
\[S_{D,C}(\F_{p^r})\simeq \coprod_{\vep\in\I_D}\coprod_{\veps}H_{\vep,\veps}(\Q)\setminus X_p(\delta)\times X^p(\gamma),\]
where $\veps$ runs over the subset of conjugacy classes of $H_\vep(\Q)/(F^\times\cap C)$ which admit a $d_0$ as in lemma 7.4 of \cite{Re1} (this implies the following sets $X_p(\delta)\neq \emptyset$), $H_{\vep,\veps}$ is the centralizer of $\veps$ in $H_\vep$, $\delta\in D^\times(\Q_{p^r})$ is the element well defined up to $\sigma$-conjugacy such that $p\delta$ is defined from $\veps$ by lemma 7.4 of loc. cit., $\gamma\in D^\times(\A_f^p)$ is defined from $\veps$ by the embedding $H_{\vep}(\A_f^p)\subset D^\times(\A_f^p)$, $X_p(\delta)$ and $X^p(\gamma)$ are defined as follows
\[\begin{split}X_p(\delta)&=\{x\in X_\vep|\,Fr_\vep^jx=\veps x\}\\
&=\{d\in D^\times(\Q_{p^r})|\,d^{-1}\delta\sigma(d)\in Y_D\}/(O_D\otimes \Z_{p^r})^\times,\\
X^p(\gamma)&=\{d\in D^\times(\A_f^p)|\,d^{-1}\gamma d\in C^p\}/C^p.\end{split}\]
\end{proposition}
\begin{proof}
This is included in proposition 7.7 of \cite{Re1}.
\end{proof}
The case for the unitary Shimura variety $Sh_{G',\ov{C}'}$ is similar. Recall that associated to a point $x\in S_{G',\ov{C}'}(\F_{p^r})$, we have the $c$-polarized virtual abelian variety with additional structures $(A,\iota,\lambda)$. The associated $p$-divisible group $H$ gives us an element $\delta\in G'(\Q_{p^r})$, well defined up to $\sigma$-conjugation by $G'(\Z_{p^r})$. It satisfies the equality $\kappa_{G'_{\Q_p}}(p\delta)=\mu^\sharp$, where $\kappa_{G'_{\Q_p}}$ is the Kottwitz map, for its definition and the $\mu^\sharp$ see \cite{Ko5}. $\delta$ only depends on the isogeny type of $x$.
\begin{proposition}
There is a bijection
\[S_{G',\ov{C}'}(\F_{p^r})\simeq \coprod_{\vep\in\I_D}\coprod_{\veps}I_{\vep,\veps}(\Q)\setminus Z_p(\delta)\times Z^p(\gamma),\]
where $\veps$ runs over the set of conjugacy classes of $I_\vep(\Q)/(Z(\Q)\cap \ov{C}')$ ($Z\subset G'$ is the center which can be viewed a subgroup of $I_\vep$), $I_{\vep,\veps}$ is the centralizer of $\veps$ in $I_\vep$, $\delta\in G'(\Q_{p^r})$ is the element well defined up to $\sigma$-conjugacy from $\veps$ as explained below, $\gamma\in G'(\A_f^p)$ is defined from $\veps$ by the embedding $I_{\vep}(\A_f^p)\subset G'(\A_f^p)$, $Z_p(\delta)$ and $Z^p(\gamma)$ are defined as follows (to define them, one needs not to know what $\delta$ and $\gamma$ are, only $\vep$ and $\veps$ will suffice for their definitions)
\[\begin{split}Z_p(\delta)&=\{x\in Z_\vep|\,Fr_\vep^jx=\veps x\}=\{x\in Z_\vep|\,(p\delta\sigma)^rx=\veps x\},\\
Z^p(\gamma)&=\{g\in G'(\A_f^p)|\,g^{-1}\gamma g\in (\ov{C}')^p\}/(\ov{C}')^p.\end{split}\]
When identifying $Z_\vep\subset G'(\mathcal{K})/G'(O_{\mathcal{K}})$, we have $Z_p(\delta)\subset G'(\Q_{p^r})/G'(\Z_{p^r})$.
\end{proposition}
\begin{proof}
By lemma 5.3 of \cite{Mi} or the method of \cite{Ko0} 1.4, this proposition follows from the description of $S_{G',\ov{C}'}(\ov{\F}_p)$ and the following facts for sufficiently small $\ov{C}'$ (which were assumed implicitly in section 4 of \cite{Ra2}):
\begin{enumerate}
\item if $\veps\in I_\vep(\Q)$ fixes a point of $Z_\vep\times G'(\A_f^p)/(\ov{C}')^p$, then $\veps\in Z(\Q)\cap \ov{C}'$;
\item $I_\vep(\Q)^{der}\cap Z(\Q)\cap \ov{C}'=\{1\}$.
\end{enumerate}
Indeed, the proof of lemma 3.6 in \cite{Lan} works here, since one only uses the property that $I_\vep(\R)$ is compact modulo center and the assumptions in the above conditions. See also lemma 5.5 of \cite{Mi}. Note for the unitary group $G'$, $Z(\Q)$ is discrete in $Z(\A_f)$. Hence for $\ov{C}'$ sufficiently small we have $Z(\Q)\cap\ov{C}'=\{1\}$. We can and we will take $C^p$ sufficiently small such that both the above two points and $Z(\Q)\cap\ov{C}'=\{1\}$ hold true.

Now we explain the $\delta$ associated to $\veps$. Let $S_{G',\ov{C}'}(\vep,\veps)$ be the term $I_{\vep,\veps}(\Q)\setminus Z_p(\delta)\times Z^p(\gamma)$, which we assume non empty. Then any point $x$ in this set will give a $\delta$ as in the paragraph before this proposition. As $S_{G',\ov{C}'}(\vep,\veps)$ is the set of isogeny class of the abelian variety with additional structures $(A,\iota,\lambda)$, when $x$ varies in this set, $\delta$ moves in its $\sigma$-conjugacy class. For those $\veps\in I_\vep(\Q)$ such that $S_{G',\ov{C}'}(\vep,\veps)=\emptyset$, we just throw them away.
\end{proof}

Actually in section 5 we will use some slightly stronger versions of the above propositions 3.2 and 3.3, in the sense that we will consider the fixed points sets of the Frobenius composed with a prime to $p$ Hecke correspondence. Then Kottwitz's method in \cite{Ko0} still works. There will be only slightly modifications for the sets $X^p(\gamma)$ and $Z^p(\gamma)$ above, cf. the proof of proposition 5.1.

\section{Test functions at $p$}
We want to introduce some test functions at $p$ which will appear in the trace formula for the related group actions on the cohomology of Shimura varieties. We will follow the method of Scholze as in \cite{Sch3}. Let the notations be as in the last section. We consider the unitary Shimura varieties first.

Recall the local reductive group
\[G'_{\Q_p}\simeq \prod_{i=1}^mD_{\nu_i}^\times\times\G_m.\]For each $1\leq i\leq m$, let $G_i$ denote the reductive group over $\Q_p$ defined by $D_{\nu_i}$. Then there are two cases: if $D$ splits at $\nu_i$, then $G_i=Res_{F_{\nu_i}|\Q_p}GL_2$; otherwise, if $D$ ramifies at $\nu_i$, $G_i$ is the inner form of $Res_{F_{\nu_i}|\Q_p}GL_2$ defined by the local quaternion algebra $D_{\nu_i}$. Let $\{\mu_i\}$ be the conjugacy class of cocharacters $\mu_i: \G_m\lra G_{i\ov{\Q}_p}$ induced from $\mu$. We have also the component $\mu_0$ of $\mu$ corresponding to the factor $\G_m$. Fix an isomorphism $\C\simeq\ov{\Q}_p$ which induces a bijection \[Hom(F,\R)=Hom(F,\C)\lra \coprod_{i=1}^mHom(F_{\nu_i},\ov{\Q}_p).\]Recall the subset $\mathcal{Z}\subset Hom(F,\R)$ of infinite places at which $D$ is split. By abuse of notation, let $\mathcal{Z}$ denote also its image under the above bijection. For each $1\leq i\leq m$, we define \[d_i=\sharp\mathcal{Z}\cap Hom(F_{\nu_i},\ov{\Q}_p),\]then clearly $d=\sum_{i=1}^md_i$ with $d=\sharp\mathcal{Z}$. One sees easily that these integers $d_i$ determine the corresponding cocharacters $\mu_i$. Now we have to make the following assumption to apply our results in \cite{Sh}:
\[\textit{for}\;1\leq i\leq m,\; \textit{if}\; D\;\textit{ramifies at}\;\nu_i,\; \textit{then}\;d_i\leq 1.\]

Let $Frob$ be a fixed geometric Frobenius in the Weil group $W_E$. Let $j\geq 1$ be a fixed integer and $r=j[\kappa_E:\F_p]$. Fix $1\leq i\leq m$. For $\tau\in Frob^jI_E\subset W_E$, $h_i\in C_c^\infty(G_i(\Z_p))$ with values in $\Q$, we have a well defined function $\phi_{\tau,h_i}\in C_c^\infty(G_i(\Q_{p^r}))$. Indeed, for the case that $D$ splits at $\nu_i$, $\phi_{\tau,h_i}$ is defined in \cite{Sch3} section 4; for the case $D$ ramifies at $\nu_i$, $\phi_{\tau,h_i}$ is defined in \cite{Sh} section 4 (if $d_i=0$ in this case, the function $\phi_{\tau,h_i}$ does not depend on $\tau$ and was denoted by $\phi_{h_i}$ in \cite{Sh}). Let $h_0\in C_c^\infty(\Z_p^\times)$ with values in $\Q$, and $\phi_{\tau,h_0}$ be the function defined in proposition 4.10 of \cite{Sch3}. Now consider the function in $C_c^\infty(G'(\Z_p))$ \[h=h_1\times\cdots\times h_m\times h_0,\]and define
\[\phi_{\tau,h}:=\phi_{\tau,h_1}\times \cdots\times \phi_{\tau,h_m}\times \phi_{\tau,h_0},\]
which is a well defined function in $C_c^\infty(G'(\Q_{p^r}))$. When $\tau, h$ as above vary, these functions $\phi_{\tau,h}$ are our test functions at $p$ for the unitary Shimura varieties $Sh_{G',\ov{C}'}$.

Now we consider the quaternion case. We have the local reductive group
\[D^\times_{\Q_p}\simeq \prod_{i=1}^mD_{\nu_i}^\times.\]Similar to the above, $D_{\nu_i}^\times$ is either $Res_{F_{\nu_i}|\Q_p}GL_2$ or the inner form of $Res_{F_{\nu_i}|\Q_p}GL_2$, depending on $D$ whether splits at $\nu_i$ or not. We will be interested in the open compact subgroups $C_p$ of $D^\times(\Q_p)$ of the form \[C_p=\prod_{i=1}^mC_{p,i},\,C_{p,i}\subset D_{\nu_i}^\times(\Q_p).\]In this case, we have also the conjugacy class of cocharacters $\{\mu\}$ defined from $h_D$ in the Shimura datum, and the local conjugacy classes $\{\mu_i\}$ for $1\leq i\leq m$. Actually these $\mu_i$ are the same as those in the unitary case. In particular we have also the integers $d_i$. To define the test functions for the Shimura varieties $Sh_{D,C}$, we will use the coarse moduli spaces $S_{G,\ov{C}}$. Recall the local reductive group
 \[G_{\Q_p}=\prod_{i=1}^m(D_{\nu_i}^\times\times F_{\nu_i}^\times).\]For each $1\leq i\leq m$, let $\wt{C_{p,i}}=C_{p,i}\times O_{F_{\nu_i}}^\times\subset D^\times_{\nu_i}\times F_{\nu_i}^\times$ and $\wt{C_p}=\prod_{i=1}^m\wt{C_{p,i}}=C_p\times F^\times(\Z_p)$.
Let $\tau\in Frob^jI_E\subset W_E$ and $Frob,\,j,r$ be as above. Let $h\in C_c^\infty(D^\times(\Z_{p}))$ with values in $\Q$. It extends to the function $\wt{h}=h\times1\in C_c^\infty(G(\Z_p))$, where the second factor is the constant function 1 on $F^\times(\Z_p)$. We fix an open compact subgroup of the form $C=C^pC_p^0\subset D^\times(\A_f^p)\times D^\times(\Q_p)$ with $C_p^0\subset D^\times(\Q_p)$ maximal.
\begin{definition}
Let $\delta\in D^\times(\Q_{p^r})$. Define
\[\phi_{\tau,h}(\delta)=0\]unless $\delta$ is associated to some $\vep\in\I_D$ and $\veps\in H_\vep(\Q)/(F^\times\cap C)$. In the latter case, let $H=A[p^\infty]$ be the $p$-divisible group associated to the abelian variety $A$ over $\ov{\F}_p$ attached to $x$ by considering the inclusion $S_{D,C}(\F_{p^r})\subset S_{D,C}(\ov{\F}_p)\subset \M_{\ov{C}}(\ov{\F}_p)\simeq S_{G,\ov{C}}(\ov{\F}_p)$. Then define
\[\phi_{\tau,h}(\delta)=tr(\tau k_D^j\times \wt{h}|H^\ast(X_{\ul{H},\wt{C_p}}\times \C_p,\Q_l)),\]
for any normal open compact pro-$p$-open subgroup $C_p\subset C_p^0$ such that $h$ is bi-$C_p$-invariant.
Here $k_D$ is the element as defined in section 2, $X_{\ul{H},\wt{C_p}}$ is the level $\wt{C_p}$ cover of the rigid generic fiber over $\mathcal{K} (=\wh{\Q_p^{nr}})$ of $Def_{\ul{H}}\simeq Spf\wh{O_{S_{G,\ov{C}},x}}$, the deformation space of $H$ as $p$-divisible group with additional structures parameterized by $\M_{G,\ov{C}}$. The tower of spaces $X_{\ul{H},\wt{C_p}}$ is equipped with a twisted Galois action of $Gal(\Q_p^{nr}/E)$, such that $\sigma\in Gal(\Q_p^{nr}/E)$ acts as $\sigma k_D^{v(\sigma)}$, where $v(\sigma)=[(\Q_p^{nr})^\sigma:E]$.
\end{definition}
Here in the above definition we have followed the convention of \cite{Sch3} that $H^\ast$ means the alternating sum of cohomology groups. We note that when $j$ is large, the twisting factor $k_D^j$ in the definition of $\phi_{\tau,h}$ disappears, since the schemes $\M_{\ov{C}}$ and $S_{G,\ov{C}}$ are isomorphic to each other after a finite unramified extension of $\Z_p$, see the definition of $\M_{\ov{C}}$ in section 3. This definition may seem somehow unnatural, but it is a direct translation of the method of \cite{Sch3} in our quaternionic setting. We will investigate these functions in more details later, and we will see that they can be defined actually in a more natural way, cf. proposition 4.3.
\begin{proposition}
The function $\phi_{\tau,h}: D^\times(\Q_{p^r})\lra \Q_l$ is well defined and takes values in $\Q$. It is locally constant with compact support. For $h=h_1\times\cdots\times h_m$ with $h_i\in C_c^\infty(D_{\nu_i}^\times(\Z_p))$ takes values in $\Q$ for $1\leq i\leq m$, we have a decomposition
\[\phi_{\tau,h}=\phi_{\tau,h_1}\times\cdots\times\phi_{\tau,h_m}\]
with some functions $\phi_{\tau,h_i}\in C_c^\infty(D_{\nu_i}^\times(\Q_{p^r}))$ defined in a similar way to $\phi_{\tau,h}$.
\end{proposition}
\begin{proof}
By construction, the twisted Galois action on $S_{G,\ov{C}}\times \Z_p^{nr}$ which defines $\M_{\ov{C}}$ coincides with the natural Galois action after base changing $S_{G,\ov{C}}$ over a finite extension of $O_E$. Therefore one sees easily that the arguments of Scholze in \cite{Sch3} section 4 apply to our situation. We only remark that as in loc. cit., the last statement is related to the decomposition of $p$-divisible groups
\[A[p^\infty]=\bigoplus_{i=1}^m(A[\pi_{\wt{\nu_i}}^\infty]\oplus A[\pi_{\wt{\nu_i}^c}^\infty]),\]
where for a place $\wt{\nu}$ of $K$ above $p$, $\pi_{\wt{\nu}}$ is a uniformizer for the local field $K_{\wt{\nu}}$, and $A$ is an ableian variety associated to an $\ov{\F}_p$-point on $S_{G,\ov{C}}$. Under the decomposition $K^\times(\Q_p)=\prod_{i=1}^m(K_{\wt{\nu_i}}^\times\times K_{\wt{\nu_i}^c}^\times)$, the element $k_D$ can be written as $(k_1,\dots,k_m)$. The functions $\phi_{\tau,h_i}$ are defined using deformation spaces of $A[\pi_{\wt{\nu_i}}^\infty]\oplus A[\pi_{\wt{\nu_i}^c}^\infty]$ and the factors $k_i$ in a similar way to $\phi_{\tau,h}$.
\end{proof}

For $j$ large, the factors $k_i^j$ disappears in the definition of $\phi_{\tau,h_i}$ in the last proposition. We would like to say more about these test functions. To compare test functions for the groups $G'$ and $D^\times$, we will introduce upper subscripts in the notation. For $D$ splits or ramifies at $\nu_i$, and the above $\tau$ and $h_i$ we have the test functions $\phi_{\tau,h_i}^{G'}$ as in the unitary case by using deformation spaces of these $p$-divisible groups with additional structures. Let $x\in S_{G,\ov{C}}(\ov{\F}_p)$ come from a point in $S_{D,C}(\F_{p^r})$ and $(A,\iota,\Lambda,\ov{\eta})$ be the associated abelian variety with additional structure. This means that $(A,\iota,\Lambda,\ov{\eta})$ and $\sigma^r(A,\iota,\Lambda,\ov{\eta k_D^j})$ define the same point of $S_{G,\ov{C}}(\ov{\F}_p)$. Let $H_i=A[\pi_{\wt{\nu_i}}^\infty]$ under the above decomposition of the associated $p$-divisible group. We write $\ul{H_i}=(H_i,\iota_i)$ with the induced action $\iota_i: O_{D_{\nu_i}}\lra End(H_i)$. Then for any open compact subgroup $C_{p,i}\subset D_{\nu_i}^\times(\Q_p)$, we have an isomorphism of rigid analytic spaces over $\mathcal{K}$
\[X_{\ul{H_i}\oplus\ul{H_i}^D,\wt{C_{p,i}}}\simeq X_{\ul{H_i},C_{p,i}},\]where both sides are cover spaces of
the generic fibers of deformation spaces of the corresponding $p$-divisible groups with additional structures. By definition, for any $\delta_i\in D^\times_{\nu_i}(\Q_p)$ coming from the $p$-divisible group $\ul{H_i}$, the test function $\phi_{\tau,h_i}^{D^\times}$ is defined by \[\begin{split}\phi_{\tau,h_i}^{D^\times}(\delta_i)&=tr(\tau k_i^j\times \wt{h_i}|H^\ast(X_{\ul{H_i}\oplus\ul{H_i}^D,\wt{C_{p,i}}}\times\C_p,\Q_l))\\&=tr(\tau (k_i')^j\times h_i|H^\ast(X_{\ul{H_i},C_{p,i}}\times\C_p,\Q_l)),\end{split}\]where in the second equality $k_i'=p^{e_i}\in K^\times_{\wt{\nu_i}}=F^\times_{\nu_i}$ is the element introduced in section 3. For $j$ large, the twisting factor $(k_i')^j$ disappears and we have
\[\phi_{\tau,h_i}^{D^\times}=\phi_{\tau,h_i}^{G'}.\]
In the general case, we have the following proposition.
\begin{proposition}
For any $\delta\in D^\times_{\nu_i}(\Q_{p^r})$ coming from a $p$-divisible group, the twisted orbital integrals of
$\phi_{\tau,h_i}^{D^\times}$ and $\phi_{\tau,h_i}^{G'}$ are the same
\[TO_{\sigma\delta}(\phi_{\tau,h_i}^{D^\times})=TO_{\sigma\delta}(\phi_{\tau,h_i}^{G'}).\]
\end{proposition}
\begin{proof}
Indeed, since $\delta$ is associated to $H_i$ which comes from the decomposition of the $p$-divisible group associated to $(A,\iota,\Lambda,\ov{\eta})\in S_{G,\ov{C}}(\ov{\F}_p)$, a $c$-virtual abelian variety with additional structures coming from a point in $S_{D,C}(\F_{p^r})$, we get $(k_i')^j\delta$ is associated to the corresponding factor $H_i'$ of the $p$-divisible group associated to $(A,\iota,\Lambda,\ov{\eta k_D^j})\in S_{G,\ov{C}}(\ov{\F}_p)$, which now is a $c$-virtual abelian variety with additional structures coming from a point in $S_{G,\ov{C}}(\F_{p^r})$. By construction, we have
\[\phi_{\tau,h_i}^{D^\times}(\delta)=\phi_{\tau,h_i}^{G'}((k_i')^j\delta).\]
As $(k_i')^j$ is an central element in $D^\times_{\nu_i}(\Q_{p^r})$, the twisted orbital integrals of $\phi_{\tau,h_i}^{G'}$ at $\delta$ and $(k_i')^j\delta$ are the same.
\end{proof}

For $1\leq i\leq m$, consider the functions $\phi_{\tau,h_i}$ defined so far in this section. The field of definition of the conjugacy class of cocharacters $\mu_i: \G_m\lra D_{\nu_i}^\times$ is denoted by $E_i$, with the associated integer $d_i$ defined in section 4. Let $r_{\mu_i}$ be the associated representation of $^L(G_{iE_i})$ where $G_i=D^\times_{\nu_i}$, cf. \cite{Ko0} 2.1.2. If $D$ splits at $\nu_i$, we can define a transfer $f_{\tau,h_i}\in C_c^\infty(G_i(\Q_p))$ with matching (twisted) orbital integrals. Moreover, by theorem 8.1 of \cite{SS} we have the following identity. For any irreducible smooth representation $\pi$ of $G_i(\Q_p)$ with $L$-parameter $\vep_\pi$,
\[tr(f_{\tau,h_i}|\pi)=tr(\tau|(r_{\mu_i}\circ\vep_{\pi}|_{W_{E_i}})|-|^{-d_i/2})tr(h_i|\pi).\]
Indeed, for the unitary case the above identity follows directly from the results of \cite{SS}. The quaternion case follows from the unitary case, by proposition 4.3.

Now we consider the case $D$ ramifies at $\nu_i$. Then the group $G_i=D_{\nu_i}^\times$ is not quasi-split. Recall we have assumed $d_i\leq 1$ in this case. If $d_i=0$, then $r_{\mu_i}$ is the trivial representation, and $h_i\in C_c^\infty(G_i(\Z_p))$ may be viewed as a transfer for the function $\phi_{\tau,h_i}$, since they have matching orbital integrals. Now we assume $d_i=1$. For $\delta\in G_i(\Q_{p^r})$, as in \cite{Sh}, the conjugacy class of the norm $N\delta:=\delta\sigma(\delta)\cdots\sigma^{r-1}(\delta)$ does not always contain an element of $G_i(\Q_p)$. Nevertheless, we have the following proposition.
\begin{proposition}
Suppose $G_i$ is defined by a local quaternion algebra. For $\delta\in G_i(\Q_{p^r})$, if the conjugacy class of the norm $N\delta$ does not contain an element of $G_i(\Q_p)$, then
\[TO_{\sigma\delta}(\phi_{\tau,h_i})=0.\]
In particular, we can define a function $f_{\tau,h_i}\in C_c^\infty(G_i(\Q_p))$ with matching (twisted) orbital integrals of $\phi_{\tau,h_i}$. Moreover, for any irreducible smooth representation $\pi$ of $G_i(\Q_p)$ with $L$-parameter $\vep_{\pi}$, we have the following character identity
\[tr(f_{\tau,h_i}|\pi)=tr(\tau|(r_{\mu_i}\circ\vep_{\pi}|_{W_{E_i}})|-|^{-1/2})tr(h_i|\pi).\]
\end{proposition}
\begin{proof}
For the unitary case, the proposition is just the special case ($n=2$) of theorem 5.4 and proposition 6.2 of \cite{Sh}. The quaternionic case follows by proposition 4.3.
\end{proof}
We remark that in the quaternionic case with $h$ as the characteristic function of $C_p^0$ divided by the volume of $C_p^0$, the test function $\phi_{\tau,h}$ should have the same twisted orbital integrals as the function $\Phi$ defined in \cite{Re1} based on the study of the local structures of these varieties. When $D$ is totally indefinite, the description of $S_{D,C}(\F_{p^r})$ and the definition of $\phi_{\tau,h}$ can be done more directly by using the PEL moduli problems for these varieties, cf. \cite{Ra1}.

\section{Cohomology and semisimple zeta functions}
In this section we change our notations sightly. We write $G=D^\times$ or $G=G'$. For $\vep\in \I_D$ and $\veps\in I_\vep(\Q)$ we write uniformly the reductive group attached to $\vep$ as $I_\vep$, the sets attached to $\veps$ as $X^p(\gamma)$ and $X_p(\delta)$. Let $l\neq p$ be prime, and $\xi$ be an algebraic $\ov{\Q}_l$-representation of $G$. Then by standard method we can associate $\ov{\Q}_l$-local systems $\mathcal{L}_\xi$ on the Shimura varieties $Sh_{G,C}$ for any open compact subgroup $C\subset G(\A_f)$. For a fixed open compact subgroup $C^p\subset G(\A_f^p)$, let $S_{G,C^p}$ be the integral model of the Shimura variety $Sh_{G,C_p^0C^p}$ defined in section 3, where $C_p^0\subset G(\Q_p)$ is as usual the maximal open compact subgroup (associated to the maximal order $O_D$ or $O_B$). We are interested in the alternating sum of cohomology groups
\[H_\xi=\sum_{i}(-1)^i\varinjlim_{C}H^i(Sh_{G,C}\times \ov{\Q}_p,\mathcal{L}_\xi)\]
as a virtual representation of $G(\A_f)\times W_E$.
To analyze this representation, we consider the traces of the action of $\tau\times hf^p$ on $H_\xi$ with $\tau\in Frob^jI_E\subset W_E, h\in C_c^\infty(G(\Z_p)), f^p\in C_c^\infty(G(\A_f^p))$. Here $G(\Z_p)=C_p^0$ is associated to the integral model of $G_{\Q_p}$ defined by $O_D$ or $O_B$.  We fix the Haar measures on $G(\Q_p)$ resp. $G(\Q_{p^r})$ that give $G(\Z_p)$ resp. $G(\Z_{p^r})$ volume 1.
\begin{proposition}
With the notations as above, we have the formula
\[tr(\tau\times hf^p|H_\xi)=\sum_{\vep\in\I_D}\sum_{\veps}\tr{vol}(I_{\vep,\veps}(\Q)\setminus I_{\vep,\veps}(\A_f))TO_{\sigma\delta}(\phi_{\tau,h})O_{\gamma}(f^p)tr\xi(\veps).\]Note by proposition 4.3, only those $\delta$ such that the conjugacy class of $N\delta$ contains an element of $G(\Q_p)$ can have non trivial contribution to the sum.
\end{proposition}
\begin{proof}
Having the description of the set $S_{G,C^p}(\F_{p^r})$ ($r=j[\kappa_E: \F_p]$), the proof is standard by applying the Lefschetz trace formula (cf. \cite{Va}). Here we sketch the main points, see also \cite{Sch3} sections 6, 7. First, fix a sufficiently small $C^p\subset G(\A_f^p)$ and a normal subgroup $C_p\subset C_p^0$.  We can assume that $f^p$ is the characteristic function of $C^pg^pC^p$ divided by the volume of $C^p$ for some $g^p\in G(\A_f^p)$, and $h$ is the characteristic function of $C_pg_p$ divided by the volume of $C_p$ for some $g_p\in G(\Q_p)$. We have the following diagram
\[\xymatrix{
&Sh_{G, C_pC^p_{g^p}}\ar[ld]_{\wt{p_1}}\ar[rd]^{\wt{p_2}}\ar[d]&\\
Sh_{G,C_pC^p}\ar[d]&Sh_{G,C_p^0C^p_{g^p}}\ar[ld]_{p_1}\ar[rd]^{p_2}&Sh_{G,C_pC^p}\ar[d]\\
Sh_{G,C_p^0C^p}& & Sh_{G,C_p^0C^p}
}\]
where $C^p_{g^p}=C^p\cap (g^p)^{-1}C^pg^p$. Here the $p_1, \wt{p_1}$ are the natural projections, whereas $\wt{p_2}$ is the composition of the natural projection $Sh_{G,C_pC^p_{g^p}}\lra Sh_{G,C_p(g^p)^{-1}C^pg^p}$ with the isomorphism $Sh_{G, C_p(g^p)^{-1}C^pg^p}\simeq Sh_{G,C_pC^p}$, and similarly for $p_2$ with $C_p$ replaced by $C_p^0$.

Let $\pi_{C_pC^p}: Sh_{G,C_pC^p}\lra Sh_{G,C_p^0C^p}$ be the natural projection. We have an isomorphism of $\ov{\Q}_l$-sheaves on $S_{G,C^p}\times\ov{\F}_p$
\[R\psi\pi_{C_pC^p\ast}\mathcal{L}_\xi\simeq \mathcal{L}_\xi\otimes R\psi\pi_{C_pC^p\ast}\ov{\Q}_l,\]
where $R\psi$ is the nearby cycle functor for $S_{G,C^p}$. Since $S_{G,C^p}$ is proper and flat, we have isomorphisms of (alternating sums of) cohomology groups
\[\begin{split}H^\ast(Sh_{G,C_pC^p}\times\ov{\Q}_p, \mathcal{L}_\xi)&=H^\ast(Sh_{G,C_p^0C^p}\times\ov{\Q}_p,\pi_{C_pC^p\ast}\mathcal{L}_\xi)\\
&=H^\ast(S_{G,C^p}\times\ov{\F}_p,R\psi\pi_{C_pC^p\ast}\mathcal{L}_\xi)\\
&=H^\ast(S_{G,C^p}\times\ov{\F}_p,\mathcal{L}_\xi\otimes R\psi\pi_{C_pC^p\ast}\ov{\Q}_l).
\end{split}\]
Now the trace
\[\begin{split}
tr(\tau\times hf^p|H_\xi)&=tr(\tau\times (g_p,g^p)_\ast|H^\ast(Sh_{G,C_pC^p}\times\ov{\Q}_p, \mathcal{L}_\xi))\\
&=tr(\tau\times (g_p,g^p)_\ast|H^\ast(S_{G,C^p}\times\ov{\F}_p,\mathcal{L}_\xi\otimes R\psi\pi_{C_pC^p\ast}\ov{\Q}_l)),\end{split}\]
where $(g_p, g^p)_\ast$ in the first equality is the map associated to (the cohomological correspondence induced from) the upper Hecke correspondence in the above diagram, and $\tau\times (g_p,g^p)_\ast$ in the second equality is the map associated to (the cohomological correspondence induced from) the composition of the Frobenius correspondence and Hecke correspondence on the special fiber. By applying the Lefschetz trace formula (theorem 2.3.2 of \cite{Va}) we get
\[tr(\tau\times hf^p|H_\xi)=\sum_{\begin{subarray}{c}x\in S_{G,C^p_{g^p}}(\ov{\F}_p)\\ Fr^j\circ p_1 (x)=p_2(x)\end{subarray}}tr(u_x)\]where $tr(u_x)$ is the local term given by the trace of
\[u_x: (\mathcal{L}_\xi\otimes R\psi\pi_{C_pC^p\ast}\ov{\Q}_l)_{Fr^j\circ p_1(x)}\lra (\mathcal{L}_\xi\otimes R\psi\pi_{C_pC^p\ast}\ov{\Q}_l)_{p_2(x)}.\]
Now similar to the description of $S_{G,C^p}(\F_{p^r})$, the set $\{x\in S_{G,C^p_{g^p}}(\ov{\F}_p)|\,Fr^j\circ p_1 (x)=p_2(x)\}$ can be written as \[\coprod_{\vep\in\I_D}\coprod_{\veps}I_{\vep,\veps}(\Q)\setminus X_p(\delta)\times X^p(\gamma),\] with \[X^p(\gamma)=\{z\in G(\A_f^p)|\,z^{-1}\gamma z\in g^pC^p\}/C^p_{g^p},\] and all the other terms are as those in section 3. We write the sum as
\[\sum_{\vep\in \I_D}\sum_{\veps}\sum_{x\in I_{\vep,\veps}(\Q)\setminus X_p(\delta)\times X^p(\gamma)}tr(u_x).\]
For the last term $tr(u_x)$, one sees it decomposes as $tr\xi(\veps)\phi_{\tau,h}(\delta)$. As \cite{Ko1} p. 432, we can rewrite the above last sum as
\[\sum_{x\in I_{\vep,\veps}(\Q)\setminus X_p(\delta)\times X^p(\gamma)}tr\xi(\veps)\phi_{\tau,h}(\delta)=\tr{vol}(I_{\vep,\veps}(\Q)\setminus I_{\vep,\veps}(\A_f))TO_{\sigma\delta}(\phi_{\tau,h})O_{\gamma}(f^p)tr\xi(\veps).\]
Hence we get the desired formula.
\end{proof}
\begin{remark}We believe this proposition holds true for general Shimura varieties, at least when we have a suitable description of the points over finite fields on them, cf. \cite{Ko1} and \cite{Sch3}. However, when passing to the automorphic side, it is not very useful unless one can prove some vanishing results for the twisted orbital integrals $TO_{\sigma\delta}(\phi_{\tau,h})$ like proposition 4.2.
\end{remark}

The main theorem of this paper is as follows. We keep our assumption for the ramified places of $D$ above $p$ as in the last section, i.e. $d_i\leq 1$. This assumption is made to apply our results in \cite{Sh} concerning the local test functions at such ramified places. The following theorem confirms the expected description of the cohomology of Shimura varieties (cf. \cite{SS}) due to Langlands and Kottwtiz in new cases. Recall the cocharacter $\mu$ associated to the Shimura data gives rise to a representation $r_{\mu}:\, ^L(G_E)\lra GL(V)$ for a finite dimensional $\ov{\Q}_l$-vector space $V$, cf. \cite{Ko0} 2.1.2 or \cite{Re1} section 11 for an explicit definition in the quaternionic case.
\begin{theorem}Under the above assumptions and notations,
we have an identity
\[H_\xi=\sum_{\pi_f}a(\pi_f)\pi_f\otimes(r_{\mu}\circ\varphi_{\pi_p}|_{W_E})|-|^{-d/2}\]
as virtual $G(\Z_p)\times G(\A_f^p)\times W_E$-representations. Here $\pi_f$ runs through irreducible admissible representations of $G(\A_f)$, the integer $a(\pi_f)$ is as in \cite{Ko2} p. 657, $\varphi_{\pi_p}$ is the local Langlands parameter associated to $\pi_p$, $d=dimSh_{G,C}$ is the dimension of the Shimura varieties $Sh_{G,C}$ for any open compact subgroup $C\subset G(\A_f)$.
\end{theorem}
\begin{proof}
This theorem is a consequence of propositions 4.4 and 5.1. In fact, by proposition 5.1 we have
\[tr(\tau\times hf^p|H_\xi)=\sum_{\vep\in\I_D}\sum_{\veps}\tr{vol}(I_{\vep,\veps}(\Q)\setminus I_{\vep,\veps}(\A_f))TO_{\sigma\delta}(\phi_{\tau,h})O_{\gamma}(f^p)tr\xi(\veps),\]for all $\tau\in Frob^jI_E\subset W_E, h\in C_c^\infty(G(\Z_p)), f^p\in C_c^\infty(G(\A_f^p))$. By proposition 4.4, only those $\delta$ such that the conjugacy class of $N\delta$ contains an element of $G(\Q_p)$ can have non trivial contribution to the sum. For these $\delta$, as in the proof of lemma 5.3 in \cite{Sh} we can find $\gamma_0$ such that $(\gamma_0;\gamma,\delta)$ forms a Kottwitz triple for $G$, and vice versa (for the quaternionic case one can also see lemma 7.4 of \cite{Re1}). In this case we have $tr\xi(\veps)=tr\xi(\gamma_0)$. In particular we can rewrite the above as
\[tr(\tau\times hf^p|H_\xi)=\sum_{(\gamma_0;\gamma,\delta)}c(\gamma_0;\gamma,\delta)O_\gamma(f^p)TO_{\delta\sigma}(\phi_{\tau,h})
tr\xi(\gamma_0),\]
where the sum runs over degree-$j$-Kottwitz triples, and $c(\gamma_0;\gamma,\delta)=\tr{vol}(I_{\vep,\veps}(\Q)\setminus I_{\vep,\veps}(\A_f))$. Then one goes through the process of pseudostabilization to get
\[tr(\tau\times hf^p|H_\xi)=N^{-1}tr(f_{\tau,h}f^p|H_\xi).\]
Here $N$ is the integer defined in \cite{Ko2} p. 659.
By Matsushima's formula (cf. lemma 4.2 of loc. cit.)
\[H_\xi=N\sum_{\pi_f}a(\pi_f)\pi_f.\]
We can take $h$ of the form $h=h_1\times\cdots\times h_m$, and $f_{\tau,h}=f_{\tau,h_1}\times\cdots\times f_{\tau,h_m}$ (for the unitary case there is also a factor $h_0$ of $h$ (resp. $f_{\tau,h_0}$ of $f_{\tau,h}$) corresponding to $\G_m$). At a place $\nu_i$ above $p$, we can apply theorem 8.1 of \cite{SS} for $f_{\tau,h_i}$ (in the quaternionic case we use the twisted version) if $D$ splits at $\nu_i$, and proposition 4.3 if $D$ is ramified at $\nu_i$. Putting all together,
we get
\[N^{-1}tr(f_{\tau,h}f^p|H_\xi)
=\sum_{\pi_f}a(\pi_f)tr(\tau|(r_{\mu}\circ\varphi_{\pi_p}|_{W_E})|-|^{-d/2})tr(hf^p|\pi_f).\]
We can conclude.

\end{proof}
As a corollary we get the local semi-simple zeta functions of these Shimura varieties. Our result generalizes the previous works of Reimann \cite{Re1, Re4} to arbitrary levels at $p$. Let $\wt{E}$ be the global reflex field and $\nu$ be a place of $\wt{E}$ above $p$ such that $E=\wt{E}_\nu$.
\begin{corollary}
Let the situation be as in the theorem. Let $C\subset G(\A_f)$ be any sufficiently small compact open subgroup. Then the semisimple local Hasse-Weil zeta function of $Sh_{G,C}$ at the place $\nu$ of $\wt{E}$ is given by
\[\zeta_\nu^{ss}(Sh_{G,C},s)=\prod_{\pi_f}L^{ss}(s-d/2, \pi_p, r_{\mu})^{a(\pi_f)dim\pi_f^C}.\]
\end{corollary}
\begin{proof}We can assume that $C$ has the form as $C^pC_p\subset G(\A_f^p)\times G(\Z_p)$. Then the corollary follows from the previous theorem and the definitions.
\end{proof}
By \cite{Ra2} section 2 one can recover the classical Hasse-Weil zeta function if one knows the Weight-Monodromy conjecture holds true in this setting.
\begin{corollary}
Assume the Weight-Monodromy conjecture is true for the Shimura varieties $Sh_{G,C}$, e.g. $d=2$. Let $C\subset G(\A_f)$ be any sufficiently small open compact subgroup in the situation of the theorem. Then the local Hasse-Weil zeta function of $Sh_{G,C}$ at the place $\nu$ of $\wt{E}$ is given by
\[\zeta_{\nu}(Sh_{G,C},s)=\prod_{\pi_f}L(s-d/2,\pi_p,r_{\mu})^{a(\pi_f)dim\pi_f^C}.\]
\end{corollary}
Recall that we have changed the notation from section 3: here $G=D^\times$ or $G=G'$. Thus the Shimura varieties $Sh_{G,C}$ are not the same as that in section 3.

Finally we remark that, along the same way, one should be able to describe the points over finite fields for the case of Shimura varieties $Sh_{G,\ov{C}}$ as in section 3, and the results of sections 4 and 5 should also be generalized to this case under the same assumptions. This will be left to the reader as an exercise.

\end{document}